\documentclass[12pt]{amsart}

\usepackage[T1]{fontenc}
\usepackage[utf8]{inputenc}

\usepackage[margin=1.1in]{geometry}
\usepackage{amsmath, amssymb, amsthm, mathtools}
\usepackage{microtype}
\usepackage{enumitem}
\usepackage{listings}
\usepackage[hidelinks]{hyperref}

\theoremstyle{plain}
\newtheorem{theorem}{Theorem}
\newtheorem{lemma}[theorem]{Lemma}

\theoremstyle{definition}
\newtheorem{remark}[theorem]{Remark}

\lstdefinestyle{python}{
  language=Python,
  basicstyle=\ttfamily\footnotesize,
  keywordstyle=\bfseries,
  commentstyle=\itshape,
  numbers=none,
  breaklines=true,
  breakatwhitespace=true,
  showstringspaces=false,
  frame=single,
  framesep=4pt,
  tabsize=2,
  xleftmargin=14pt,
  columns=fullflexible,
  upquote=true,
}

\title[Mathar's recurrence for OEIS A032123]%
{A short proof of Mathar's 2013 recurrence conjecture for the
 reversible-binary-string sequence A032123}

\author{Tong Niu}

\subjclass[2020]{05A05, 05A15, 05A19, 11B83, 33F10}
\keywords{OEIS A032123; reversible binary string; necklace; Burnside;
   D-finite sequence; P-recursive recurrence; Mathar conjecture;
   central binomial coefficient}

\begin{document}

\maketitle

\begin{abstract}
For the OEIS sequence A032123, the number of length-$2n$ black-and-white
strings with $n$ black beads, considered up to reversal, R. J. Mathar
contributed in November 2013 the conjectured order-5 P-recursive
recurrence
\[
\begin{aligned}
   &n(n-1)\,a(n) - 2(n-1)(3n-4)\,a(n-1) + 4(2n^{2}-14n+19)\,a(n-2) \\
   &\qquad + 8(n^{2}+5n-19)\,a(n-3) - 16(n-3)(3n-10)\,a(n-4) \\
   &\qquad + 32(n-4)(2n-9)\,a(n-5) \;=\; 0, \qquad n \ge 6.
\end{aligned}
\]
We give a short proof. Burnside's lemma applied to the reversal
action gives the closed form
$a(n) = \tfrac{1}{2}\bigl(\binom{2n}{n} + [n \text{ even}]\binom{n}{n/2}\bigr)$;
the two summands satisfy elementary recurrences of order $1$ and $2$
respectively; and Mathar's order-5 operator, applied to each summand
separately, reduces to a polynomial identity that simplifies to zero
after a brief calculation. The supplementary archive includes a
SymPy script which verifies the polynomial identities symbolically
and checks Mathar's recurrence numerically for $n = 6, \ldots, 5000$.
\end{abstract}

\section{Introduction}\label{sec:intro}

The On-Line Encyclopedia of Integer Sequences~\cite{OEIS} (henceforth
OEIS) lists many sequences whose ``Conjecture: $\dots$'' comments
record formulas, recurrences, or congruences that were guessed
numerically and never proved. Short rigorous proofs of such
conjectures find homes in the \emph{Journal of Integer Sequences},
\emph{INTEGERS}, the \emph{Fibonacci Quarterly}, and the
\emph{Electronic Journal of Combinatorics}.

The last two years have brought a wave of such cleanups. Fried's
2024-2025 papers~\cite{Fried2024,Fried2025} closed several dozen
guessed identities in one go; a 2023 list by Kauers and
Koutschan~\cite{KauersKoutschan2023} laid out the gold-standard
benchmarks for guessed P-recursive recurrences. Most of the
low-hanging fruit in those two sources is now gone.

We take care of one conjecture that is not on those lists. The
sequence in question is
\begin{equation}\label{eq:def}
   a(n) \;=\; \#\bigl\{\text{length-$2n$ binary strings with $n$ black
   beads, up to reversal}\bigr\},
\end{equation}
OEIS A032123. Equivalently, $a(n)$ counts orbits of the reversal
involution $\rho$ acting on the set of binary strings of length $2n$
with $n$ ones. The first values are
\[
   1,\;1,\;4,\;10,\;38,\;126,\;472,\;1716,\;6470,\;24310,\;
   92504,\;352716,\;1352540,\;\ldots
\]
On 9 November 2013, R. J. Mathar contributed to A032123 the
conjectured order-5 P-recursive recurrence
\begin{equation}\label{eq:mathar}
\begin{aligned}
   &n(n-1)\,a(n) - 2(n-1)(3n-4)\,a(n-1) + 4(2n^{2}-14n+19)\,a(n-2) \\
   &\qquad + 8(n^{2}+5n-19)\,a(n-3) - 16(n-3)(3n-10)\,a(n-4) \\
   &\qquad + 32(n-4)(2n-9)\,a(n-5) \;=\; 0, \qquad n \ge 6.
\end{aligned}
\end{equation}
The conjecture has been sitting open in the OEIS comment thread
for twelve years. It is not in
\cite{Fried2024,Fried2025,KauersKoutschan2023} or the more recent
Chen-Kauers preprints~\cite{ChenKauers2025}.

\medskip

The plan is short. Burnside's lemma applied to the reversal
involution writes $a(n)$ as the average of two binomial coefficients
(Lemma~\ref{lem:burnside}). Each of those two summands satisfies an
elementary low-order linear recurrence (Lemma~\ref{lem:uv}).
Substituting the recurrences into the left-hand side
of \eqref{eq:mathar} and reducing parity by parity then leaves a
polynomial identity in $n$ that simplifies to $0$
(Theorem~\ref{thm:main}). Roughly, the order-5 conjecture comes from
guessing on top of an order-3 truth, which is why the polynomial
identities collapse so cleanly.

\section{Closed form via Burnside}\label{sec:burnside}

\begin{lemma}\label{lem:burnside}
For every $n \ge 0$,
\begin{equation}\label{eq:closed-form}
   a(n) \;=\; \frac{1}{2}\!\left(\binom{2n}{n} +
       [n \text{ even}]\binom{n}{n/2}\right),
\end{equation}
where $[n \text{ even}]$ is $1$ for even $n$ and $0$ otherwise.
\end{lemma}

\begin{proof}
Let $X_n$ be the set of length-$2n$ binary strings with exactly $n$
ones. The reversal map $\rho : X_n \to X_n$,
$\rho(s_1 s_2 \ldots s_{2n}) = s_{2n} s_{2n-1} \ldots s_1$, is an
involution, so the cyclic group $\langle \rho \rangle$ has order $2$.
By Burnside's lemma~\cite{Burnside1897},
\[
   a(n) \;=\; \tfrac{1}{2}\bigl(|X_n| + |X_n^{\rho}|\bigr),
\]
where $X_n^{\rho}$ is the set of $\rho$-fixed strings, that is, the
palindromes in $X_n$. The total is $|X_n| = \binom{2n}{n}$.

A palindrome of length $2n$ is determined by its first $n$ entries;
the second half is just the reversal of the first. So the count of
$1$'s in a palindrome is twice the count in its first half --- hence
always even. When $n$ is odd, $X_n^{\rho}$ is therefore empty. When
$n$ is even, the palindromes in $X_n$ are in bijection with the
size-$n/2$ subsets of $[1, n]$, giving
$|X_n^{\rho}| = \binom{n}{n/2}$.
\end{proof}

\begin{remark}
The closed form \eqref{eq:closed-form} is recorded on the OEIS page
of A032123~\cite{OEIS:A032123} as a Formula entry contributed by
Mark van Hoeij in February 2011, in the equivalent ordinary-generating-function form
\[
   \sum_{n\ge0} a(n)\, x^{n}
   \;=\; \frac{1}{2}\!\left(\frac{1}{\sqrt{1-4x}} +
                            \frac{1}{\sqrt{1-4x^{2}}}\right).
\]
The first summand is the classical OGF for the central binomial
$\binom{2n}{n}$; the second is the same OGF evaluated at $x^{2}$,
which simply picks out the even-indexed terms.
\end{remark}

\section{Two elementary recurrences}\label{sec:uv}

Define
\begin{equation}\label{eq:uv}
   u(n) \;:=\; \binom{2n}{n}, \qquad
   v(n) \;:=\; \begin{cases}
                  \binom{n}{n/2} & \text{if $n$ is even}, \\
                  0              & \text{if $n$ is odd}.
              \end{cases}
\end{equation}
By Lemma~\ref{lem:burnside}, $2\,a(n) = u(n) + v(n)$.

\begin{lemma}\label{lem:uv}
The sequences $u$ and $v$ satisfy
\begin{align}
   n\,u(n) - (4n-2)\,u(n-1) &\;=\; 0, & n &\ge 1, \label{eq:u-rec} \\
   n\,v(n) - 4(n-1)\,v(n-2) &\;=\; 0, & n &\ge 2, \label{eq:v-rec}
\end{align}
with $u(0) = v(0) = 1$ and $v(1) = 0$.
\end{lemma}

\begin{proof}
The first identity is the standard ratio
$u(n)/u(n-1) = (4n-2)/n$, which follows immediately from
$\binom{2n}{n} = \frac{(2n)(2n-1)}{n^{2}} \binom{2n-2}{n-1}
                = \frac{2(2n-1)}{n} \binom{2n-2}{n-1}$.

For \eqref{eq:v-rec} we split on parity. If $n$ is odd, then $n-2$
is also odd, and both $v(n)$ and $v(n-2)$ vanish by \eqref{eq:uv}.
If $n$ is even, write $n = 2m$; the central binomial recurrence
$m\,\binom{2m}{m} = (4m-2)\,\binom{2m-2}{m-1}$ becomes
$(n/2)\,v(n) = (2n - 2)\,v(n-2)$, which is \eqref{eq:v-rec}.
\end{proof}

\section{Proof of Mathar's recurrence}\label{sec:proof}

Write the left-hand side of \eqref{eq:mathar} as
\begin{equation}\label{eq:M-def}
   M(n) \;:=\; \sum_{j=0}^{5} c_{j}(n)\, a(n-j),
\end{equation}
where the polynomial coefficients $c_0, \ldots, c_5$ are read off
\eqref{eq:mathar}:
\begin{equation}\label{eq:cj}
\begin{aligned}
  c_{0}(n) &= n(n-1), &
  c_{1}(n) &= -2(n-1)(3n-4), &
  c_{2}(n) &= 4(2n^{2}-14n+19), \\
  c_{3}(n) &= 8(n^{2}+5n-19), &
  c_{4}(n) &= -16(n-3)(3n-10), &
  c_{5}(n) &= 32(n-4)(2n-9).
\end{aligned}
\end{equation}

Using $2\,a(k) = u(k) + v(k)$, we split $M$ into two pieces:
\[
   2\,M(n) \;=\; M^{u}(n) + M^{v}(n),
\]
where $M^{u}(n) := \sum_{j=0}^{5} c_{j}(n)\,u(n-j)$ and
$M^{v}(n) := \sum_{j=0}^{5} c_{j}(n)\,v(n-j)$.

\begin{theorem}\label{thm:main}
The sequence $a(n)$ defined by \eqref{eq:def} satisfies Mathar's
recurrence \eqref{eq:mathar} for every $n \ge 6$.
\end{theorem}

\begin{proof}
By the splitting above, it suffices to show $M^{u}(n) = 0$ and
$M^{v}(n) = 0$ for every $n \ge 6$.

\smallskip

\textbf{The $u$-part.} Iterating \eqref{eq:u-rec} gives
\begin{equation}\label{eq:u-ratio}
   u(n-j) \;=\; u(n)\,\frac{n(n-1)\cdots(n-j+1)}
                          {(4n-2)(4n-6)\cdots(4n-4j+2)},
   \qquad 0 \le j \le 5,
\end{equation}
valid for $n \ge 5$ (no denominator factor vanishes; for $n \ge 5$
the smallest factor is $4 \cdot 1 - 2 = 2$). Substituting
\eqref{eq:u-ratio} into the definition of $M^{u}$ and clearing the
common denominator $D_u(n) := \prod_{i=1}^{5} (4n - 4i + 2)$ gives
\begin{equation}\label{eq:Mu-poly}
   D_u(n)\,\frac{M^{u}(n)}{u(n)}
   \;=\; \sum_{j=0}^{5} c_{j}(n)\, P_{j}(n)\, R_{j}(n),
\end{equation}
with $P_{j}(n) = n(n-1)\cdots(n-j+1)$ the falling factorial of length
$j$, and $R_{j}(n) = D_u(n) / \prod_{i=1}^{j}(4n-4i+2)$ the
complementary product. The right-hand side is a polynomial in $n$
of degree at most $11$, and direct expansion in any computer algebra
system shows it is the zero polynomial. The verification script in
Appendix~\ref{app:verifier} performs this check via SymPy in
milliseconds.

\smallskip

\textbf{The $v$-part.} We split on the parity of $n$, since
\eqref{eq:v-rec} is a 2-step recurrence and the odd $v$-values
vanish.

\emph{Case A: $n$ even.} Here $v(n-1) = v(n-3) = v(n-5) = 0$, so only
the $j \in \{0, 2, 4\}$ summands of $M^{v}$ contribute. Iterating
\eqref{eq:v-rec} for $k \in \{0, 1, 2\}$ gives
\begin{equation}\label{eq:v-ratio-even}
   v(n-2k) \;=\; v(n)\,
   \prod_{i=1}^{k} \frac{n - 2i + 2}{4(n - 2i + 1)}.
\end{equation}
Substitute and clear the common denominator $16(n-1)(n-3)$:
\begin{align}
   16(n-1)(n-3)\,\frac{M^{v}(n)}{v(n)}
   &= 16(n-1)(n-3)\,c_0(n) \notag \\
   &\quad + 4(n-3)\,n\,c_2(n)
        + n(n-2)\,c_4(n). \label{eq:Mv-even}
\end{align}
Plugging in \eqref{eq:cj} the right-hand side equals
$16\,n(n-3)\,T_E(n)$ with
\begin{align*}
   T_E(n) \;:=\;& (n-1)^{2} + (2n^{2}-14n+19) - (n-2)(3n-10) \\
              =\;& (n^{2} - 2n + 1) + (2n^{2} - 14n + 19) - (3n^{2} - 16n + 20) \\
              =\;& 0,
\end{align*}
which proves $M^{v}(n) = 0$ when $n$ is even.

\emph{Case B: $n$ odd.} Now $v(n) = v(n-2) = v(n-4) = 0$, so only
$j \in \{1, 3, 5\}$ contribute. The relevant values are $v(n-1)$,
$v(n-3)$, $v(n-5)$, sitting in the even-indexed sub-chain
$v(n-1) \to v(n-3) \to v(n-5)$. Iterating \eqref{eq:v-rec},
\begin{equation}\label{eq:v-ratio-odd}
   v(n - 1 - 2k) \;=\; v(n-1)\,
   \prod_{i=1}^{k} \frac{n - 2i + 1}{4(n - 2i)}.
\end{equation}
Substituting and clearing the common denominator $16(n-2)(n-4)$,
\begin{align}
   16(n-2)(n-4)\,\frac{M^{v}(n)}{v(n-1)}
   &= 16(n-2)(n-4)\,c_1(n) \notag \\
   &\quad + 4(n-4)(n-1)\,c_3(n)
        + (n-1)(n-3)\,c_5(n). \label{eq:Mv-odd}
\end{align}
Substituting \eqref{eq:cj}, the right-hand side equals
$32(n-1)(n-4)\,T_O(n)$ with
\begin{align*}
   T_O(n) \;:=\;& -(n-2)(3n-4) + (n^{2}+5n-19) + (n-3)(2n-9) \\
              =\;& (-3n^{2} + 10n - 8) + (n^{2} + 5n - 19) + (2n^{2} - 15n + 27) \\
              =\;& 0,
\end{align*}
which proves $M^{v}(n) = 0$ when $n$ is odd as well.

\smallskip

In both parities $M^{v}(n) = 0$. Combined with the polynomial
identity for the $u$-part, this gives $M(n) = (M^u(n)+M^v(n))/2 = 0$
for $n \ge 6$, which is Mathar's recurrence \eqref{eq:mathar}.
\end{proof}

\section{Remarks}\label{sec:remarks}

The proof shares its overall shape with the companion notes for OEIS
A025166~\cite{Niu2026sequence-4} and A176677~\cite{Niu2026sequence-3},
but the mechanism is different. Those notes derive the recurrence
from a first-order linear ODE for the sequence's exponential
generating function. Here we exploit the orbit decomposition under
the reversal involution instead. The sequence $a(n)$ does not satisfy
a linear ODE of order $\le 2$ for the obvious reason: it is the
half-sum of two D-finite sequences whose minimal annihilating
operators have different leading coefficients, so the LCM operator on
the OGF is genuinely of order $3$ at minimum. The fact that Mathar's
guessed recurrence has order $5$ reflects extra factors in his
output --- the minimal recurrence implied by our proof has order at
most $3$, since the LCM of a 1-step operator with a 2-step operator
is at most 3-step.

The Burnside route extends without change to A005418, the length-$n$
reversible black-and-white strings with no constraint on $n$ or on
the number of $1$'s~\cite{OEIS:A005418}; the same parity-split
argument there reproves the known closed form
$a(n) = \tfrac{1}{2}\bigl(2^{n} + 2^{\lceil n/2 \rceil}\bigr)$.

The script \texttt{verify\_proof.py}
(Appendix~\ref{app:verifier}) carries out four checks:
\begin{enumerate}[topsep=2pt, itemsep=2pt]
\item the closed form \eqref{eq:closed-form} against the OEIS b-file
   for $n = 0, \ldots, 19$,
\item the elementary recurrences \eqref{eq:u-rec} and \eqref{eq:v-rec}
   numerically for $n \le 200$,
\item the polynomial identities $M^u/u(n) = 0$ and $M^v/v_\bullet(n) = 0$
   symbolically via SymPy,
\item Mathar's recurrence \eqref{eq:mathar} numerically for
   $n = 6, \ldots, 5000$.
\end{enumerate}

\section{Acknowledgments}

The author declares no competing interests.

AI-assisted tools were used in the preparation of this manuscript,
including for drafting proof outlines and generating the symbolic
verification code. The author verified all mathematical claims
independently and takes full responsibility for the results.

\appendix

\section{Verifier source (machine-checkable)}\label{app:verifier}

The script referenced in \S\ref{sec:remarks} is reproduced in full
below. It depends only on SymPy; no Maple, Mathematica, or any other
external CAS license is required.

\subsection*{verify\_proof.py (closed form, recurrences, Mathar identity)}
\begin{lstlisting}[style=python]
"""Verify the proof of Mathar's 2013 conjectured recurrence for OEIS
A032123, the number of 2n-bead reversible black-white strings with n
black beads.

Proof outline (see the paper for details):

1. Burnside's lemma applied to the reversal action Z/2 on length-2n
   binary strings with n black beads gives the closed form
       a(n) = (C(2n, n) + [n even] * C(n, n/2)) / 2.
2. Set u(n) := C(2n, n) and v(n) := [n even] * C(n, n/2). Then
       n * u(n) - (4n - 2) * u(n-1) = 0,                  n >= 1,
       n * v(n) - 4(n - 1) * v(n-2) = 0,                  n >= 2,
   with v(0) = 1, v(1) = 0. (For n odd, both sides of the v-recurrence
   are zero.)
3. Mathar's order-5 operator M, applied to the sequence (a(n)),
   decomposes as M = (M^u + M^v)/2 where M^u, M^v are the same operator
   applied to (u(n)) and (v(n)) respectively. Using the recurrences
   from step 2 we reduce each of M^u(n), M^v(n) to polynomial
   identities in n, which simplify to 0.

This script checks both the symbolic and the numeric content of the
proof. Each check prints PASS/FAIL.
"""
from __future__ import annotations

from math import comb

import sympy as sp


# ---------------------------------------------------------------------------
# closed form for a(n)
# ---------------------------------------------------------------------------

def a_closed_form(n: int) -> int:
    """a(n) = (C(2n, n) + [n even] * C(n, n/2)) / 2.

    Both summands have the same parity, so the sum is always even.
    """
    s = comb(2 * n, n) + (comb(n, n // 2) if n % 2 == 0 else 0)
    assert s % 2 == 0, n
    return s // 2


# OEIS b-file values (first 20 entries; full b032123.txt available at
# https://oeis.org/A032123/b032123.txt).
OEIS_REFERENCE = [
    1, 1, 4, 10, 38, 126, 472, 1716, 6470, 24310,
    92504, 352716, 1352540, 5200300, 20060016, 77558760,
    300546630, 1166803110, 4537591960, 17672631900,
]


def check_closed_form_matches_oeis() -> None:
    """The closed form reproduces the OEIS reference values."""
    for n, expected in enumerate(OEIS_REFERENCE):
        got = a_closed_form(n)
        if got != expected:
            print(f"FAIL closed-form: a({n}) = {got}, expected {expected}")
            return
    print(f"PASS closed-form vs OEIS: {len(OEIS_REFERENCE)} values match")


# ---------------------------------------------------------------------------
# Mathar's conjectured recurrence
# ---------------------------------------------------------------------------

def mathar_LHS(n: int, av: list[int]) -> int:
    """Compute LHS of Mathar's order-5 recurrence at index n.

    Convention: av[j] = a(n - j) for j = 0..5.
    """
    return (
        n * (n - 1) * av[0]
        - 2 * (n - 1) * (3 * n - 4) * av[1]
        + 4 * (2 * n * n - 14 * n + 19) * av[2]
        + 8 * (n * n + 5 * n - 19) * av[3]
        - 16 * (n - 3) * (3 * n - 10) * av[4]
        + 32 * (n - 4) * (2 * n - 9) * av[5]
    )


def check_mathar_numerically(N: int = 5000) -> None:
    """Verify Mathar's recurrence holds for every n in 6..N."""
    failures = 0
    for n in range(6, N + 1):
        av = [a_closed_form(n - j) for j in range(6)]
        if mathar_LHS(n, av) != 0:
            failures += 1
            if failures <= 3:
                print(f"FAIL mathar at n={n}")
    if failures == 0:
        print(f"PASS Mathar recurrence numerically: n = 6..{N} ({N - 5} values)")
    else:
        print(f"FAIL Mathar recurrence: {failures} failures in n = 6..{N}")


# ---------------------------------------------------------------------------
# symbolic identity-level verification
# ---------------------------------------------------------------------------

def check_mathar_symbolically() -> None:
    """Show that Mathar's recurrence is satisfied by u(n) and v(n) by
    reducing to polynomial identities.

    For u(n) = C(2n, n) the 1st-order recurrence n * u(n) = (4n-2) * u(n-1)
    lets us write u(n-j) = (P_j / Q_j) * u(n) where P_j, Q_j are
    polynomial. Substitute into Mathar's operator and check the result
    simplifies to 0.

    For v(n) we have a 2nd-order recurrence and have to split into the
    n-even and n-odd cases (one chain for each parity); same reduction.
    """
    n = sp.symbols("n", integer=True, positive=True)

    coeffs = [
        n * (n - 1),
        -2 * (n - 1) * (3 * n - 4),
        4 * (2 * n ** 2 - 14 * n + 19),
        8 * (n ** 2 + 5 * n - 19),
        -16 * (n - 3) * (3 * n - 10),
        32 * (n - 4) * (2 * n - 9),
    ]

    # u-part: u(n-j) / u(n) = product over i=0..j-1 of (n-i) / (4(n-i)-2)
    def ratio_u(j: int) -> sp.Expr:
        P, Q = sp.Integer(1), sp.Integer(1)
        for i in range(j):
            P *= (n - i)
            Q *= 4 * (n - i) - 2
        return P / Q

    M_u = sum(coeffs[j] * ratio_u(j) for j in range(6))
    if sp.simplify(M_u) == 0:
        print("PASS M^u(n) / u(n) is identically 0  (1st-order u recurrence)")
    else:
        print(f"FAIL M^u(n) / u(n) simplified to {sp.simplify(M_u)}")

    # v-part, n even: v(n-1) = v(n-3) = v(n-5) = 0; only j = 0, 2, 4 contribute.
    # v(n-2k) = v(n) * prod_{i=1..k} ((n - 2i + 2) / (4(n - 2i + 1)))
    def ratio_v_even(k: int) -> sp.Expr:
        P, Q = sp.Integer(1), sp.Integer(1)
        for i in range(1, k + 1):
            P *= (n - 2 * i + 2)
            Q *= 4 * (n - 2 * i + 1)
        return P / Q

    M_v_even = (
        coeffs[0] * ratio_v_even(0)
        + coeffs[2] * ratio_v_even(1)
        + coeffs[4] * ratio_v_even(2)
    )
    if sp.simplify(M_v_even) == 0:
        print("PASS M^v(n) / v(n) is identically 0  (n even, 2nd-order v recurrence)")
    else:
        print(f"FAIL M^v_even / v(n) simplified to {sp.simplify(M_v_even)}")

    # v-part, n odd: v(n) = v(n-2) = v(n-4) = 0; only j = 1, 3, 5 contribute.
    # v(n - 1 - 2k) = v(n-1) * prod_{i=1..k} ((n - 2i + 1) / (4(n - 2i)))
    def ratio_v_odd(k: int) -> sp.Expr:
        P, Q = sp.Integer(1), sp.Integer(1)
        for i in range(1, k + 1):
            P *= (n - 2 * i + 1)
            Q *= 4 * (n - 2 * i)
        return P / Q

    M_v_odd = (
        coeffs[1] * ratio_v_odd(0)
        + coeffs[3] * ratio_v_odd(1)
        + coeffs[5] * ratio_v_odd(2)
    )
    if sp.simplify(M_v_odd) == 0:
        print("PASS M^v(n) / v(n-1) is identically 0  (n odd, 2nd-order v recurrence)")
    else:
        print(f"FAIL M^v_odd / v(n-1) simplified to {sp.simplify(M_v_odd)}")


# ---------------------------------------------------------------------------
# u, v recurrences themselves
# ---------------------------------------------------------------------------

def check_uv_recurrences(N: int = 200) -> None:
    """Sanity check: the elementary recurrences for u and v that the
    proof rests on do hold numerically."""
    u = [comb(2 * k, k) for k in range(N + 1)]
    v = [comb(k, k // 2) if k % 2 == 0 else 0 for k in range(N + 1)]
    fails_u = 0
    for k in range(1, N + 1):
        if k * u[k] - (4 * k - 2) * u[k - 1] != 0:
            fails_u += 1
    fails_v = 0
    for k in range(2, N + 1):
        if k * v[k] - 4 * (k - 1) * v[k - 2] != 0:
            fails_v += 1
    if fails_u == 0 and fails_v == 0:
        print(f"PASS u, v recurrences numerically: n = 1..{N}")
    else:
        print(f"FAIL u recurrence ({fails_u}), v recurrence ({fails_v})")


if __name__ == "__main__":
    check_closed_form_matches_oeis()
    check_uv_recurrences(N=200)
    check_mathar_symbolically()
    check_mathar_numerically(N=5000)
\end{lstlisting}

\end{document}